\documentstyle[11pt,amssymb,amsmath,amscd,amsbsy,amsfonts,theorem,hyperref]{article}

\input xy
\xyoption{all}

\newcommand{\zz}{{\Bbb Z}}
\newcommand{\zzz}{{\mathbf{Z}}}
\newcommand{\nn}{{\Bbb N}}

\newcommand{\rr}{{\Bbb R}}

\newcommand{\aaa}{{\Bbb A}}

\newcommand{\hh}{{\Bbb H}}
\newcommand{\ttt}{{\Bbb T}}

\newcommand{\ddim}{\operatorname{dim}}

\newcommand{\ddet}{\operatorname{det}}

\newcommand{\Homi}{\underline{\operatorname{Hom}}}
\newcommand{\Hom}{\operatorname{Hom}}
\newcommand{\op}[1]{\operatorname{#1}}
\newcommand{\kbar}{\overline{k}}

\newcommand{\ffi}{\varphi}

\newcommand{\la}{\langle}
\newcommand{\ra}{\rangle}
\newcommand{\lva}{\langle\!\langle}   % @!=1/10*\!   negative space
\newcommand{\rva}{\rangle\!\rangle}   % @,=1/10*\,   positive space
\newcommand{\row}{\rightarrow}

\newcommand{\low}{\leftarrow}
\newcommand{\lrow}{\longrightarrow}
\renewcommand{\leq}{\leqslant}
\renewcommand{\geq}{\geqslant}

\newcommand{\calm}{{\cal M}}
\newcommand{\hm}{\operatorname{H}_{\calm}}

\newcommand{\nichego}[1]{}

\newcommand{\un}[1]{\underline{#1}}
\newcommand{\wt}[1]{\widetilde{#1}}

\newcommand{\cn}{{\cal N}}
\newcommand{\cq}{{\cal Q}}
\newcommand{\Ch}{\operatorname{Ch}}

\newcommand{\dmk}{\op{DM}(k)}
\newcommand{\dmgmk}{\op{DM}_{gm}(k)}
\newcommand{\dmgmkD}{\op{DM}_{gm}(k;\zz/2)}
\newcommand{\dmgmkF}[1]{\op{DM}_{gm}(k;#1)}
\newcommand{\dmkD}{\op{DM}(k;\zz/2)}

\newcommand{\dmkF}[1]{\op{DM}(k;#1)}
\newcommand{\hii}{{\cal X}}
\newcommand{\whii}{\widetilde{{\cal X}}}
\newcommand{\DQMgm}{\op{DQM}^{gm}}
\newcommand{\Kbt}{K^b(Tate(\zz/2))}

\newcommand{\dmkQD}[1]{\op{DM}_{#1}(k;\zz/2)}

\newcommand{\dmkQF}[2]{\op{DM}_{#1}(k;#2)}
\newcommand{\dmELD}[2]{\op{DM}({#1}/{#2};\zz/2)}

\newcommand{\shk}{{\cal SH}(k)}

\newcommand{\picd}{\op{Pic}(\dmgmkD)}
\newcommand{\picq}{\op{Pic}_{qua}}
\newcommand{\chqD}{Chow_{qua}(k,\zz/2)}

\newcommand{\Qed}{\hfill$\square$\smallskip}
\newenvironment{proof}{\noindent{\it Proof}:}{\vskip 5mm}

\newtheorem{prop}{Proposition}[section]{\bf}{\it}
\newtheorem{thm}[prop]{Theorem}{\bf}{\it}
\newtheorem{lem}[prop]{Lemma}{\bf}{\it}
{\bf}{\it}
\newtheorem{defi}[prop]{Definition}{\bf}{\it}
{\bf}{\it}
{\bf}{\it}
\newtheorem{exa}[prop]{Example}{\bf}{\it}
\newtheorem{rem}[prop]{Remark}{\bf}{}
\newtheorem{que}[prop]{Question}{\bf}{\it}

{\bf}{\it}
{\bf}{\it}
\newtheorem{cor}[prop]{Corollary}{\bf}{\it}
{\bf}{\it}

\begin{document}

\title{Affine quadrics and the Picard group of the motivic category}
\author{Alexander Vishik\footnote{School of Mathematical Sciences, University
of Nottingham}}
\date{}

\maketitle

\begin{abstract}
In this article we study the subgroup of the Picard group of Voevodsky's category of geometric motives $\dmgmkD$
generated by the reduced motives of affine quadrics. Our main tools here are the functors of Bachmann - \cite{BQ},
but we also provide an alternative method. We show that the group in question can be described in terms of
indecomposable direct summands in the motives of projective quadrics over $k$.
In particular, we describe all the relations among the reduced motives of affine quadrics. We also extend the Criterion
of motivic equivalence of projective quadrics.
\end{abstract}

\section{Introduction}

The study of the Picard group of the motivic category in the algebro-gemetric context was initiated by Po Hu in \cite{Hu}, who considered
the case of the $\aaa^1$-stable homotopy category of Morel-Voevodsky. It was established there that the reduced classes of affine Pfister
quadrics $\{\lva a_1,\ldots,a_r\rva=b\}$ of small fold-ness represent invertible objects in $\shk$, and some relations among these
classes in $\op{Pic}(\shk)$ were found. It was conjectured that the same should hold for arbitrary $r$.

The topic was picked up by T.Bachmann in \cite{BQ}. Here, instead of the $\aaa^1$-stable homotopic category $\shk$, the Voevodsky's category
of motives $\dmkD$ was considered. This simplified the task somewhat. As a result, not only the conjectures of Po Hu were proven in this
context, but it was shown that the reduced motive $\widetilde{M}(A_q)$ of any affine quadric $A_q=\{q=1\}$ is invertible in $\dmkD$.
This was established with the
help of functors $\Phi^E$ of Bachmann. These tensor triangulated functors, defined for every finitely generated field extension
$E/k$, map the tensor triangulated category $\DQMgm$ generated by the motives of smooth projective quadrics to the category $\Kbt$
of bi-graded $\zz/2$-vector spaces. They are characterized by the property that $\Phi^E(T(i)[j])=T(i)[j]$ (the 1-dimensional vector
space of the specified bi-degree), where $T$ is the monoidal unit, while $\Phi^E(M(Q))=0$, for every projective quadric $Q/k$ which
stays anisotropic over $E$.
It was shown in \cite{BQ} that the collection of these functors (for all finitely generated $E/k$) is conservative,
detects invertible objects, and is injective on the $\op{Pic}$.

In \cite{BV} it was proven that the map: $q\mapsto\widetilde{M}(A_q)$ defines an embedding of sets:
$GW(k)\hookrightarrow\op{Pic}(\dmkD)$ of the Grothendieck-Witt ring of quadratic forms (or, by the result of F.Morel
\cite{morel2004motivic-pi0}, of $\pi_{(0)[0]}^s({\Bbb{S}})$) into the Picard group of the motivic category.

In the current paper, we study the relations among these elements in $\op{Pic}$.
Or, which is the same, the subgroup $\picq$ of $\op{Pic}(\dmkD)$ generated by the reduced motives of affine quadrics.
It appears that these can be described in terms of motives of projective quadrics and the direct sum operation.
Inside $\picq$ there is a subgroup $\ttt\cong\zz\oplus\zz$ consisting of Tate-motives $T(i)[j]$. It is enough to describe $\picq/\ttt$.
This group is generated by our (shifted) reduced motives $e^q:=\widetilde{M}(A_q)[1]$ of affine quadrics.
First of all, in Proposition \ref{inverse}, we complement the invertibility result of T.Bachmann by observing that our set of generators
is closed under inverses:
$(e^q)^{-1}=e^{q'}$ in $\picq/\ttt$, where $q'=\la 1\ra\perp -q$
(note, that the operation $q\mapsto q'$ is a "square root" of $q\mapsto q\perp\hh$).
In particular, this gives that $e^{\lva\alpha\rva}$ is the inverse of the reduced Rost motive $\widetilde{M}_{\alpha}$
(here $\alpha$ is a pure symbol in $K^M_*(k)/2$).
In Theorem \ref{Dedekind} we provide a large supply of linearly independent elements in $\picq$. Namely, the collection $\{e^{q_i}\}_i$
will be linearly independent as long as all the projective quadrics $Q'_i$ are anisotropic and pair-wise not stably bi-rationally equivalent.
Moreover, it is shown in Proposition \ref{conditional-sta} that if the Question \ref{old-question} has positive answer, then a maximal
such collection will form a $\zz$-basis of $\picq/\ttt$ (note, that the fact that this group is torsion free follows from \cite{BQ}).
Every smooth projective quadric $Q$ can be cut into affine ones (using some flag of plane sections).
Multiplying the respective elements $e^q$ we
get the new element $\ddet(Q)$. From some basic relations among $e^q$'s it follows that this does not depend on the choice of a flag, and
is an invariant of $Q$, and even of the motive of $Q$. The set $\{\ddet(Q)\}_Q$, where $Q$ runs over all smooth projective quadrics over $k$
provides another set of generators of $\picq$.
In Theorem \ref{relations} we establish all the relations among these elements in $\picq/\ttt$. Namely,
$\prod_i\ddet(P_i)=\prod_j\ddet(Q_j)\in\picq/\ttt$ if and only if
$\oplus_iM(P_i)$ and $\oplus_jM(Q_j)$ are {\it Tate-equivalent}, i.e. if we ignore the Tate-summands in both, then the respective
indecomposable (anisotropic) direct summands of both hand sides can be identified up to Tate-shift.
This embeds $\picq$ into the free abelian group with the basis consisting of indecomposable direct summands in the motives of $k$-quadrics
considered up to Tate-shift. What is remarkable here is that the question about Voevodsky's triangulated motives and the tensor product
operation is reduced to the one about classical Chow motives and the direct sum operation.
As a small by-product we can complement the classical Criterion of motivic equivalence of projective
quadrics (\cite{IMQ,lens}, see also \cite{Kar2}) with the equality $\ddet(P)=\ddet(Q)\in\picq$.
Thus, in $\picq$ we have two generating subsets: one identified with the isomorphism classes of quadratic forms, another with the
isomorphism classes of motives of projective quadrics.

All the above results are obtained with the help of the Bachmann's functors $\Phi^E$ which provide a very effective tool for comparing
elements of $\picq$. In Section \ref{alternative} we introduce an alternative method which permits to perform the same calculations.
Here we use the \v{C}ech simplicial schemes and some ideas from \cite{IMQ}. The idea is very simple: for a smooth projective $P$, the motive
$\hii_P$ of the \v{C}ech simplicial scheme is an idempotent: $\hii_P^{\otimes 2}\cong\hii_P$, and so is it's "complement"
$\whii_P=\op{Cone}(\hii_P\row T)$. As a result, we get two orthogonal projections $\otimes\hii_P$ and $\otimes\whii_P$ on $\dmk$ which
define a semi-orthogonal decomposition of this category. For different varieties, these projectors naturally commute, and we can
consider a poly-semi-orthogonal decomposition corresponding to a finite collection ${\mathbf X}=\{X_i\}_{i\in I}$
of smooth projective varieties. The resulting functor is obviously conservative and so, detects invertible geometric objects.
It is shown in Proposition \ref{inj-pic-XJ} that it is also injective on the $\op{Pic}$.
It follows from the results of \cite{IMQ} that, for any object $A$ of $\DQMgm$, there is an appropriate collection ${\mathbf X}$,
for which all the
projections of $A$ will be extensions of Tate-motives.
And, if $A$ represents an element of $\picq$, then there is a collection, where all the projections are the Tate-motives $T(i)[j]$.
In particular, we re-prove the Bachmann's result on the invertibility of the reduced motives of affine quadrics - see Proposition
\ref{eq-invert}.
Moreover, two elements of $\picq$ are equal if and only if the respective functions $(i)[j]$ on the set of (non-trivial) projectors
are the same. This creates an environment which permits to substitute the functors of Bachmann in the study of $\picq$.
The new approach is not restricted to the subcategory $\DQMgm$ only, but permits to study the whole $\op{Pic}$ of $\dmk$.
We will address this question in a sequel to this paper.

\section{Motives of affine quadrics}

\subsection{Notations and some basic facts}

Let $k$ be a field of characteristic different from $2$, and $q$ be a (non-degenerate)
quadratic form of dimension $n$ over $k$. We denote as
$A_q$ the affine quadric $\{q=1\}$. Then $A_q$ can be considered as a (not necessarily split) sphere. In particular, over $\kbar$,
the motive $M(A_q)_{\kbar}=T\oplus T([n/2])[n-1]$ is a sum of just two Tate-motives. This motive is a complete invariant of $q$ -
see \cite[Theorem 2.1]{BV}. We have a natural projection $A_q\row\op{Spec}(k)$, and it was shown in \cite{BV} that the reduced motive $\widetilde{M}(A_q)=\op{Cone}[-1](M(A_q)\row T)$ of $A_q$ determines $q$ as well.
This reduced motive is a {\it form} of a Tate-motive, as over $\kbar$ it becomes isomorphic to $T([n/2])[n-1]$.  It
belongs to the category $\dmgmkD$ of geometric motives of Voevodsky - see \cite{voevodsky-triang-motives}.
Moreover, it was shown by T.Bachmann in \cite{BQ} that this motive is invertible there, i.e. it represents an element of $\picd$.
And for $p=q\perp\hh$, one has $\widetilde{M}(P)\cong\widetilde{M}(Q)(1)[2]$ - \cite[Lemma 34]{BQ}. Hence, we get an embedding
\begin{equation*}
\begin{split}
&GW(k) \hookrightarrow  \picd\\
&q-r\hh \mapsto \widetilde{M}(A_q)(-r)[-2r+1]
\end{split}
\end{equation*}
of sets of the Grothendieck-Witt ring of quadratic forms into the Picard of the category of geometric motives.
In other words, we get a complete invariant of the $(0)[0]$-stable $\aaa^1$-homotopy group of spheres (as, by the result of
F.Morel \cite[Theorem 6.2.1]{morel2004motivic-pi0} this group coincides with the $GW(k)$). The corresponding map in topology is
the map $\zzz=\pi_0^s({\Bbb{S}})\row\op{Pic}(D(Ab))$ sending $n$ to $T[n]$ which happens to be an isomorphism.
The aim of the current paper is to study the motivic variant of such a map.

We start by introducing some notations. Let us denote as $e^q\in\picd$ the shifted reduced motive $\widetilde{M}(A_q)[1]$ and as
$\picq$ the subgroup of $\picd$ generated by $e^q$, for all quadratic forms $q/k$.

For a quadratic form $q$, let $q'$ be the quadratic form $\la 1\ra\perp -q$, and $Q,Q'$ be the respective smooth projective quadrics.
Then $A_q=Q'\backslash Q$, and we have the Gysin triangle:
\begin{equation}
\label{AQ'Q}
 M(Q')\row M(Q)(1)[2]\row M(A_q)[1]\row M(Q')[1].
\end{equation}
When both quadratic forms $q$ and $q'$ are split, $M(Q')$ and $M(Q)$ are sums of (pure) Tate-motives, which implies that
$M(A_q)=T\oplus T([n/2])[n-1]$ and $e^q=T([n/2])[n]$ (one can also see it from \cite[Lemma 34]{BQ}). In particular,
$\picq$ contains a subgroup $\ttt\cong\zz\times\zz$ consisting of Tate-motives $T(i)[j],\,i,j\in\zz$, which coincides with the whole
$\picq$ for an algebraically closed field (since all quadrics are split there).
Restriction to $\kbar$ together with the projection provide an isomorphism
$\picq\stackrel{\cong}{\lrow}\ttt\times(\picq/\ttt)$.
Thus, the description of $\picq$ is reduced to that of
$\picq/\ttt$, which amounts to describing the relations among $e^q$'s there.

First, we will describe the inverse of $e^q$.

\begin{prop}
\label{inverse}
Let $q'=\la 1\ra\perp -q$. Then in $\picq/\ttt$,
$$
(e^q)^{-1}=e^{q'}.
$$
\end{prop}

\begin{proof}
Let $q''=\la 1\ra\perp -q'=\hh\perp q$ and $\ddim(q)=n$.
Considering the $\op{Cone}[-1]$ of the map of triangles
$$
\xymatrix{
M(A_q) \ar[d] \ar[r] & M(Q') \ar[d] \ar[r] & M(Q)(1)[2] \ar[d]^{0} \ar[r] & M(A_q)[1] \ar[d] \\
T \ar[r] & 0 \ar[r] & T[1] \ar@{=}[r] & T[1],
}
$$
where $M(A_q)\row T$ is the standard projection, we obtain a distinguished triangle
$$
\xymatrix{
\wt{M}(A_q) \ar[r] & M(Q') \ar[r]^(0.4){(a,b)} & M(Q)(1)[2]\oplus T \ar[r]^(0.55){(c,d)} & \wt{M}(A_q)[1],
}
$$
where $c$ is the unique lifting of the map $M(Q)(1)[2]\row M(A_q)[1]$ from (\ref{AQ'Q}),
$d$ is the canonical map from the definition of $\wt{M}(A_q)$, $a$ is a map from (\ref{AQ'Q}), and from the diagram chase
one can see that the standard projection $M(Q')\row T$ factors through $b$, which means that these two maps coincide.
The same applies to the pair $Q'\subset Q''$. We get exact triangles (after shifting):
\begin{align*}
&M(Q')\row M(Q)(1)[2]\oplus T\row \widetilde{M}(A_q)[1]\row M(Q')[1].\\
&M(Q'')\row M(Q')(1)[2]\oplus T\row \widetilde{M}(A_{q'})[1]\row M(Q'')[1].
\end{align*}
Since $q''=\hh\perp q$, $Q''$ is isotropic, and $Q$ can be identified with the quadric of lines $l$ on $Q''$ passing through
a fixed rational point $p$. This gives the decomposition $M(Q'')=T\oplus M(Q)(1)[2]\oplus T(n)[2n]$, where the map $M(Q)(1)[2]\row M(Q'')$
is given by the cycle $A=\{(l,x)|x\in l\}\subset Q\times Q''$. The map $M(Q'')\row M(Q')(1)[2]$ from the Gysin triangle is dual of the
embedding $M(Q')\row M(Q'')$ and given by the cycle $B=\Delta_{Q''}\cap (Q'\times Q'')$. Taking $p$ outside $Q'\subset Q''$,
we obtain that the composition
$M(Q)(1)[2]\row M(Q'')\row M(Q')(1)[2]$  is given by the cycle $C=\Delta_{Q'}\cap (Q\times Q')$, where the embedding $Q\subset Q'$ is
given by the choice of $p$. In other words, this composition is dual of the embedding $M(Q)\row M(Q')$.
The map $T(n)[2n]\row M(Q'')$ is given by the generic cycle of $Q''$, so the composition $T(n)[2n]\row M(Q'')\row M(Q')(1)[2]$ is
given by the generic cycle of $Q'$ and, hence, is dual of the projection $M(Q')\row T$. Thus, the resulting map
$M(Q)(1)[2]\oplus T(n)[2n]\row M(Q')(1)[2]$ is dual to the map $M(Q')\row M(Q)(1)[2]\oplus T$, and we obtain that
$e^{q'}=\Homi(e^q,T(n)[2n+1])$, where $\Homi(-,-)$ is the internal $\Hom$ in $\dmgmkD$.
By \cite[Theorem 33]{BQ} (see also Proposition \ref{eq-invert} below), $e^q$ is an invertible object.
It follows from the standard properties of duality that the dual of an
invertible object is the inverse of it. Thus,
$$
e^q\cdot e^{q'}=T(n)[2n+1]\in\picq.
$$
\Qed
\end{proof}

\begin{exa}
Let $\alpha\in K^M_*(k)/2$ be some pure symbol, and $\lva\alpha\rva$ be the respective Pfister form.
Then $\lva\alpha\rva=q_{\alpha}=\la 1\ra\perp -\widetilde{q}_{\alpha}$, and the motive of the
affine quadric $Q_{\alpha}\backslash\widetilde{Q}_{\alpha}$ is the Rost-motive $M_{\alpha}$ \cite{R2}.
Hence, in $\picq/\ttt$,  $e^{\lva\alpha\rva}=(e^{\widetilde{q}_{\alpha}})^{-1}$ is the inverse of the reduced Rost-motive
$\widetilde{M}_{\alpha}=\op{Cone}[-1](M_{\alpha}\row T)$.
\end{exa}

\subsection{The functors of Bachmann}

In \cite{BQ} T.Bachmann considers $\DQMgm$ - the thick tensor triangulated subcategory of $\dmgmkD$ generated by motives of smooth
projective quadrics over $k$. Then, for any field extension $E/k$ he constructs a tensor triangulated functor:
$$
\Phi^E:\DQMgm\lrow\Kbt,
$$
where $\Kbt$ is the category of finite-dimensional bi-graded $\zz/2$-vector spaces (which we can view as direct sums of
Tate-motives $T(i)[j]$). This functor is essentially defined by the following two properties:
\begin{itemize}
\item[$1)$ ] $\Phi^E(T(i)[j])=T(i)[j]$;
\item[$2)$ ] If smooth projective quadric $Q_E$ is anisotropic, then $\Phi^E(Q)=0$.
\end{itemize}

The main result of $\cite{BQ}$ is:

\begin{thm} {\rm (Bachmann, \cite[Theorem 31]{BQ})}
\label{MT-B}
The collection of functors $\{\Phi^E\}$ for all finitely generated extensions
$E/k$ is conservative, and it is injective on the Picard.
\end{thm}

Since $\op{Pic}(\Kbt)=\zz\times\zz$, in particular, this implies:

\begin{prop} {\rm (Bachmann, \cite[Corollary 32]{BQ})}
\label{no-torsion}
The group $\picq$ has no torsion.
\end{prop}

The above functors of Bachmann will represent the main tool in our calculations.

\section{Structure of $\picq$}
\label{structure-picq}

\subsection{Linearly independent elements}

We will identify the elements of $\op{Pic}(\Kbt)$ with the Tate-motives $T(i)[j]$ (with identification given by $\Phi^E$).
It follows from the results of Bachmann \cite{BQ} that, for any quadratic form $q/k$ and any extension $E/k$, the value
of $\Phi$ on $e^q$ is a single Tate motive $T(f(q,E))[g(q,E)]$. Thus, we get two functions: $(q,E)\mapsto f(q,E),g(q,E)\in\zz$.
We can identify targets of various $\Phi^E$'s, and since $\Phi^E(e^q)$ is invertible, we can consider expressions like
$\frac{\Phi^E}{\Phi^F}(e^q)$ which is still a single Tate-motive.

We would like to describe the group $\picq/\ttt$. Since $\Phi^E$ maps $\ttt$ isomorphically to the $\op{Pic}(\Kbt)$, this quotient-group
is still torsion-free.
Here is a large supply of linearly independent elements there.

\begin{thm}
\label{Dedekind}
Let $\{q_i\}_{i\in I}$ be a collection of quadratic forms over $k$, s.t. $q'_i$ is anisotropic, for all $i$, and for $i\neq j$, the
forms $q'_i$ and $q'_j$ are not stably bi-rationally equivalent. Then the collection of elements $\{e^{q_i}\}_{i\in I}$ is
linearly independent in $\picq/\ttt$.
\end{thm}

\begin{proof}
Recall, that two quadrics $P$ and $Q$ are stably bi-rationally equivalent if and only if there are rational maps $P\dashrightarrow Q$
and $Q\dashrightarrow P$ - \cite[Theorem X.4.25]{Lam}.
Suppose, we have some linear relation in $\picq$:
$$
\prod_i(e^{q_i})^{m_i}=T(*)[*'].
$$
Consider a directed graph whose vertices are $q'_i$ (for $q_i$ appearing in the above equation), and where we have an arrow
$q'_i\row q'_j$ if and only if there exists a rational map $Q'_i\dashrightarrow Q'_j$ (note, that this condition just means that
$Q'_j|_{k(Q'_i)}$ has a rational point, or in other words, that $i_W(q'_j|_{k(Q'_i)})>0$).
Since this property is transitive (by the valuative criterion of properness - \cite[Theorem II.4.7]{Har}),
and all our forms $q'_i$ are pairwise not stably birationally equivalent, we obtain that
our graph has no oriented cycles. Hence, there is (at least one) final vertex $q'_l$. Consider $F_l=k(Q'_l)$.
Then $i_W(q'_j|_{F_l})=0$, for $j\neq l$ (as the vertex is final), while $i_W(q'_l|_{F_l})\neq 0$
(because any quadric is isotropic over its own function field).
Since the forms $q'_j$, and so $q_j$, for $j\neq l$, stay anisotropic over $F_l$, it follows that $\Phi^{F_l}$ acts in the same way on
$e^{q_j}$ as $\Phi^{k}$. At the same time, since $q'_l$ is anisotropic over $k$ and isotropic over $F_l$, these functors act differently
on $e^{q_l}$ - see \cite[Sect. 3]{BV}, or Proposition \ref{phiE-on-ePQ} below. As a result, we obtain that
$$
\frac{\Phi^{F_l}}{\Phi^k}\left(\prod_i(e^{q_i})^{m_i}\right)=T(x\cdot m_l)[y\cdot m_l],\hspace{3mm}\text{where}\hspace{2mm}T(x)[y]\neq T.
$$
And this must be equal to $T$, since all $\Phi^E$'s act the same (identical) way on Tate-motives.
Hence, $m_l=0$, and we managed to exclude
one term from our relation. Then we argue by induction.
\Qed
\end{proof}

\begin{cor}
\label{Pfister-indep}
Let $\{\lva\alpha\rva\}_{\alpha}$ be the collection of all Pfister forms (of various fold-ness) for all non-zero pure symbols
$\alpha\in K^M_*(k)/2$.
Then the collection $\{e^{\lva\alpha\rva}\}_{\alpha}$ is linearly independent in $\picq/\ttt$.
\end{cor}

\begin{proof}
Indeed, let $\lva\alpha\rva=q_{\alpha}=\la 1\ra\perp -\widetilde{q}_{\alpha}$, where $\widetilde{q}_{\alpha}$ is a {\it pure part} of
a Pfister form. Then the collection $\{\widetilde{q}_{\alpha}\}_{\alpha}$ satisfies the conditions of the Theorem \ref{Dedekind},
since different Pfister forms are not stably bi-rationally equivalent (because the existence of a rational map
$Q_{\alpha}\dashrightarrow Q_{\beta}$ means that $Q_{\beta}|_{k(Q_{\alpha})}$ is isotropic and hence hyperbolic (being a Pfister form),
which implies that
$q_{\beta}$ is divisible by $q_{\alpha}$ - see \cite[Corollary 23.6]{EKM}).
Finally, by Proposition \ref{inverse},
$e^{\lva\alpha\rva}=(e^{\widetilde{q}_{\alpha}})^{-1}$.
\Qed
\end{proof}

\subsection{The new generators}

Let $Q\supset P$ be a co-dimension one embedding of smooth projective quadrics. We will use the
notation $e^{Q\backslash P}$ for the (shifted) reduced motive of the affine quadric $Q\backslash P$.

Let us explicitly describe the value of the functor $\Phi^E$ on $e^{Q\backslash P}$ in terms of the Witt-indices of
both projective quadrics - cf. \cite[Sect. 3]{BV}. Below we will use the additive notation $(x)[y]$ for the elements of the abelian
group $\zz^2$.

\begin{prop}
\label{phiE-on-ePQ}
Let $P'\supset P$ be a co-dimension one embedding of smooth projective quadrics, $\ddim(P')=m'$, $\ddim(P)=m$ (of course, $m'=m+1$),
$E/k$ be some field extension and $j_{P'}=i_W(P'_E)$,
$j_P=i_W(P_E)$ be the Witt indices of $P'$ and $P$ over $E$. Then $\Phi^E(e^{P'\backslash P})=T(x)[y]$, where
$(x)[y]=(f(P')-f(P))[g(P')-g(P)]$, for some functions $f$ and $g$. More precisely,
$$
(x)[y]=\sum_{l'=0}^{j_{P'}-1}(m'-2l')[2m'-4l'+1]-\sum_{l=0}^{j_P-1}(m-2l)[2m-4l+1].
$$
\end{prop}

\begin{proof}
We have an exact triangle
$$
M(P')\row M(P)(1)[2]\oplus T\row \widetilde{M}(P'\backslash P)[1]\row M(P')[1].
$$
Our Witt indices are related as follows: $j_{P}\leq j_{P'}\leq j_{P}+1$ (since $p\subset p'\subset p\perp\hh$).
From the defining property of Bachmann's functors (as well as from \cite{BQ}) we see that $\Phi^E(e^{P'\backslash P})$ will be a single
Tate-motive $T(x)[y]$ whose grading depends only on the above Witt indices.
It remains to determine the exact shape of such a dependence.
If $(j_{P'},j_P)=(l,l)$, then the "non-cancelled" Tate-motive is on the $P$-side and $(x)[y]=(l)[2l]$, while if
$(j_{P'},j_P)=(l+1,l)$, then the "non-cancelled" Tate-motive is on the $P'$-side and $(x)[y]=(m'-l)[2m'-2l+1]$.
Consider $(x)[y]$ as a function of $(l',l)$ (for $l\leq l'\leq l+1$) defined by these formulas.
We can move from the pair $(0,0)$ to $(j_{P'},j_{P})$ in $j_{P'}+j_{P}$ steps: $(0,0)\row(1,0)\row(1,1)\row(2,1)\row\ldots$.
When we move $(l',l')\row(l'+1,l')$, $(x)[y]$ jumps up by $(m'-2l')[2m'-4l'+1]$.
When we move $(l+1,l)\row(l+1,l+1)$, then  $(x)[y]$ jumps down by $(m-2l)[2m-4l+1]$.
Finally, for $(l',l)=(0,0)$,  $(x)[y]=(0)[0]$.
Hence, the formula.
\Qed
\end{proof}

\begin{prop}
\label{PQRS}
Suppose, we have co-dimension one embeddings
$$
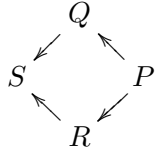
\xymatrix @-1.2pc{
& Q \ar[dl] & \\
S & & P \ar[dl] \ar[ul] \\
& R \ar[ul] &
}
$$
of smooth projective quadrics. Then in $\picq$,
$$
e^{S\backslash Q}\cdot e^{Q\backslash P}=e^{S\backslash R}\cdot e^{R\backslash P}.
$$
\end{prop}

\begin{proof}
Let $E/k$ be some field extension, and $j_S,j_R,j_Q,j_P$ be the Witt indices of our quadrics over $E$.
Then, by Proposition \ref{phiE-on-ePQ}, we have that both $\Phi^E(e^{S\backslash Q}\cdot e^{Q\backslash P})$
and $\Phi^E(e^{S\backslash R}\cdot e^{R\backslash P})$ are isomorphic to $T(x)[y]$, where
$$
(x)[y]=\sum_{i=0}^{j_{S}-1}(\ddim(S)-2i)[2\ddim(S)-4i+1]-\sum_{l=0}^{j_P-1}(\ddim(P)-2l)[2\ddim(P)-4l+1].
$$
By the Theorem \ref{MT-B}, $e^{S\backslash Q}\cdot e^{Q\backslash P}=e^{S\backslash R}\cdot e^{R\backslash P}$
in $\picq$.
\Qed
\end{proof}

The above relations among $e^q$ permit to introduce new generators of $\picq$.

\begin{defi}
\label{def-det}
Let $Q$ be an $m$-dimensional smooth projective quadric with the complete flag of subquadrics:
$Q=Q_m\supset Q_{m-1}\supset\ldots\supset Q_0$. Define:
$$
\ddet(Q):=e^{Q_m\backslash Q_{m-1}}\cdot e^{Q_{m-1}\backslash Q_{m-2}}\cdot\ldots\cdot e^{Q_0}\in\picq.
$$
\end{defi}

Clearly, one can express the shifted reduced motive $e^{Q\backslash P}$ of an affine quadric $Q\backslash P$ as
$\ddet(Q)/\ddet(P)$. Thus, determinants of smooth projective quadrics is another system of generators of $\picq$.
It follows from Proposition \ref{PQRS} that $\ddet(Q)$ does not depend on the choice of a complete flag in $Q$
and is an invariant of $Q$. Moreover, it actually depends on $M(Q)$ only.

\begin{prop}
\label{Phidet-MQ}
Let $Q$ be a smooth projective quadric of dimension $m$. Then:
\begin{itemize}
\item[$1)$ ] For any $E/k$, $\Phi^E(\ddet(Q))=T(x)[y]$, where
$$
(x)[y]=\sum_{i=0}^{i_W(Q_E)-1}(m-2i)[2m-4i+1].
$$
\item[$2)$ ] $\ddet(Q)$ depends on $M(Q)$ only.
\end{itemize}
\end{prop}

\begin{proof}
1) Follows straight from the Definition \ref{def-det}
and Proposition \ref{phiE-on-ePQ}.

2) It follows from part 1) that $\Phi^E(\ddet(Q))$ depends only on $m$ and $i_W(Q_E)$. By the Criterion of motivic equivalence of
projective quadrics - \cite[Prop 5.1]{IMQ}, or \cite[Thm 4.18]{lens} (see also \cite{Kar2}), the motives $M(P)$ and $M(Q)$
of two smooth projective quadrics are isomorphic if and only if $\ddim(P)=\ddim(Q)$ and $i_W(P_E)=i_W(Q_E)$, for any
field extension $E/k$. Thus, $\Phi^E(\ddet(Q))$ depends only on $M(Q)$, and by Theorem \ref{MT-B}, so does $\ddet(Q)$ itself.
\Qed
\end{proof}

Let $N$ be a direct summand of the (possibly shifted) motive $M(Q)(i)[2i]$ of anisotropic quadric.
Then, over $\kbar$, it splits into a direct sum of pure
Tate-motives $T(l)[2l]$. These Tate-motives are of two kinds: the {\it lower} and the {\it upper} ones. The lower ones are characterized
by the property that the splitting map $N\row T(l)[2l]$ is defined already over the ground field $k$, while for the upper ones, the
splitting map $T(l)[2l]\row N$ is defined over $k$ - see \cite[Appendix]{EC} for details.
Similarly, for an extension $E/k$, we denote as
$Tate^{up}(N_E)$ the collection of {\it upper} Tate-motives splitting from $N_E$, and as $Tate_{lo}(N_E)$ - the collection of
{\it lower} Tate-motives splitting from $N$ over $E$.
Let us define certain auxiliary elements of $\op{Pic}(\Kbt)$ we will use.

\begin{defi}
\label{def-PhiEk-detN}
Let $N$ be a direct summand of $M(Q)(i)[2i]$ with $Q$ anisotropic. Define:
$$
\frac{\Phi^E}{\Phi^k}(\ddet(N)):=\left(\otimes_{T(l)[2l]\in Tate^{up}(N_E)}T(l)[2l]\right)\otimes
\left(\otimes_{T(l)[2l]\in Tate_{lo}(N_E)}T(l)[2l-1]\right)^{-1}.
$$
\end{defi}

If $N$ is a direct summand of the motive of a possibly isotropic quadric $Q$, then we can define
$\frac{\Phi^E}{\Phi^k}(\ddet(N))$ by the same formula, where we ignore those Tate-summands of $N_E$ which already
split over $k$. Note, that the number of upper Tate motives splitting from $N_E$ is equal to the number of the lower Tate-motives
splitting there (by \cite[Thm 4.19]{lens}), for any anisotropic $N$. Hence, the Tate shift of $N$ does not affect the result:
$$
\frac{\Phi^E}{\Phi^k}(\ddet(N(i)[2i]))=\frac{\Phi^E}{\Phi^k}(\ddet(N)).
$$
Our elements behave multiplicatively with respect to the direct sum of motives:
$$
\frac{\Phi^E}{\Phi^k}(\ddet(N_1\oplus N_2))=
\frac{\Phi^E}{\Phi^k}(\ddet(N_1))\otimes\frac{\Phi^E}{\Phi^k}(\ddet(N_2)).
$$
And, consequently, it can be extended to arbitrary direct sums of direct summands as above.

In the case of $N=M(Q)$ with $\ddim(Q)=m$, we have the decomposition over $E$:
$$
M(Q_E)=\oplus_{i=0}^{i_W(Q_E)-1}(T_{lo}(i)[2i]\oplus T^{up}(m-i)[2m-2i])\oplus M((Q_E)_{anis})(i_W(Q_E))[2i_W(Q_E)].
$$
From Proposition \ref{Phidet-MQ}(1) we get:
$$
\frac{\Phi^E}{\Phi^k}(\ddet(M(Q)))=\frac{\Phi^E}{\Phi^k}(\ddet(Q)).
$$
This explains the notation. A'priori, it is unclear, if the element $\frac{\Phi^E}{\Phi^k}(\ddet(N))$ comes from some element
"$\ddet(N)$" in $\picq$
(so, the notations are somewhat misleading). But, in certain cases, we can produce $\ddet(N)\in\picq$.

\begin{exa}
\label{Pfister-decomp}
Let $\alpha\in K^M_r(k)/2$ be a pure symbol, and $\lva\alpha\rva$ be the respective Pfister form. Then
$\lva\alpha\rva=q_{\alpha}=\la 1\ra\perp -\widetilde{q}_{\alpha}$, the Gysin triangle
$$
M(Q_{\alpha})\row M(\widetilde{Q}_{\alpha})(1)[2]\row M(Q_{\alpha}\backslash\widetilde{Q}_{\alpha})[1]\row M(Q_{\alpha})[1]
$$
is split, and the motive $M(Q_{\alpha}\backslash\widetilde{Q}_{\alpha})$ is the Rost-motive $M_{\alpha}$.
Also, $e^{Q_{\alpha}\backslash\widetilde{Q}_{\alpha}}=\ddet(Q_{\alpha})/\ddet(\widetilde{Q}_{\alpha})$.
Since $M(Q_{\alpha})=M(\widetilde{Q}_{\alpha})(1)[2]\oplus M_{\alpha}$, we obtain that
$$
\frac{\Phi^E}{\Phi^k}(\ddet(Q_{\alpha})/\ddet(\widetilde{Q}_{\alpha}))=\frac{\Phi^E}{\Phi^k}(\ddet(M_{\alpha})).
$$
So, we can define $\ddet(M_{\alpha})\in\picq$ as
$\ddet(Q_{\alpha})/\ddet(\widetilde{Q}_{\alpha})=e^{Q_{\alpha}\backslash\widetilde{Q}_{\alpha}}$.
It is nothing else, but the shifted reduced Rost-motive $\widetilde{M}_{\alpha}[1]$.
By the result of Rost \cite{R,R2},
$$
M(Q_{\alpha})=\oplus_{i=0}^{2^{r-1}-1}M_{\alpha}(i)[2i].
$$
Due to the multiplicativity property of the $\frac{\Phi^E}{\Phi^k}(\ddet(-))$ and Bachmann's injectivity result - Theorem \ref{MT-B}
(comparing also $\Phi^k$ of both parts),
we obtain:
$$
\ddet(Q_{\alpha})=\left(e^{Q_{\alpha}\backslash\widetilde{Q}_{\alpha}}\right)^{2^{r-1}}.
$$
In particular, by Proposition \ref{inverse}, in $\picq/\ttt$ we have an identity:
$$
\ddet(Q_{\alpha})=\left(e^{\lva\alpha\rva}\right)^{-2^{r-1}}.
$$
\end{exa}

Consider the full additive subcategory $\chqD$ of $Chow(k,\zz/2)$ which is the pseudo-abelian envelope of the subcategory
generated by the motives of smooth projective quadrics. Then in this category holds the Krull-Schmidt principle,
i.e. the decomposition into irreducible objects is unique - see \cite{lens} and \cite{CM}.
Let us introduce the following equivalence relation on the set of objects of $\chqD$.

\begin{defi}
\label{T-equiv}
Suppose $N$ and $M$ are objects of $\chqD$. We say that $N\stackrel{T}{\sim}M$ if anisotropic indecomposable direct summands of $N$ can
be identified up to (reordering and) Tate-shifts with such summands of $M$. More precisely, if $N\cong(\oplus Tates)\oplus \oplus_{i=1}^rN_i$,
$M\cong(\oplus Tates)\oplus \oplus_{i=1}^rM_i$, where $M_i\cong N_i(a_i)[2a_i]$, for some $a_i\in\zz$ and some choice of ordering.
\end{defi}

In particular, the $T$-equivalence ignores Tate-motives, but it keeps the total rank of anisotropic direct summands, and it is stable
under field extensions.

We will use a minor modification of Bachmann's injectivity Theorem \ref{MT-B}.
Observe that, for any $E/k$, the map $\frac{\Phi^E}{\Phi^k}:\picq/\ttt\row\op{Pic}(\Kbt)$ is well defined.

\begin{prop}
\label{inj-T}
The collection of maps $\frac{\Phi^E}{\Phi^k}$, for all finitely generated $E/k$, is injective on $\picq/\ttt$.
\end{prop}

\begin{proof}
Suppose all the homomorphisms $\frac{\Phi^E}{\Phi^k}$ vanish on certain element $x\in\picq$. That means that
$\Phi^E(x)=\Phi^k(x)=T(a)[b]$, for any $E/k$, and some fixed $a,b\in\zz$. But $\Phi^E(T(a)[b])$ is also equal to $T(a)[b]$.
By the Bachmann's injectivity Theorem, $x=T(a)[b]$.
\Qed
\end{proof}

Let us recall some facts about indecomposable direct summands in the motives of quadrics.
Suppose, $N$ is such a summand in the motive of a quadric $Q$, s.t. the smallest Tate-motive in the decomposition of $N_{\kbar}$
is $T(n)[2n]$. Then we can assign to $N$ the Grassmannian $X_N=G(Q,n)$ of $n$-dimensional projective subspaces on $Q$, and to the latter
variety we can assign the motive $\hii_{X_N}$ of the respective \v{C}ech simplicial scheme $\check{C}ech(X_N)$ (where all our motives are with $\zz/2$-coefficients). Recall that this is a
{\it form} of a Tate-motive which becomes isomorphic to Tate-motive if and only if our variety has a zero-cycle of degree 1 - see
\cite[Sect. 2.3]{IMQ}. For quadratic Grassmannians the latter condition is equivalent to the existence of a rational point (by Springer's
theorem).

\begin{prop}
\label{dir-sum-facts}
With notations as above, let $N$ and $L$ be indecomposable direct summands in the motives of quadrics. Then the following are equivalent:
\begin{itemize}
\item[$(1)$] $N$ is isomorphic to $L$ up to Tate-shift;
\item[$(2)$] $X_N$ and $X_L$ are stably bi-rationally equivalent;
\item[$(3)$] $\hii_{X_N}\cong\hii_{X_L}$ (in this case, this isomorphism is unique).
\end{itemize}
\end{prop}

\begin{proof}
The fact that the lowest Tate-motive $T(n)[2n]$ splits from $N$ is equivalent to the fact that it splits from the motive of the
respective quadric $Q$. This, in turn, is equivalent to the fact that $i_W(Q)>n$ (by \cite[Prop. 1]{R} and \cite[Prop. 2.6]{lens}),
which means exactly that the Grassmannian $X_N$ has a
rational point. Thus (1) implies that $X_N$ and $X_L$ have rational points simultaneously, i.e. there are rational maps both ways.
For quadratic Grassmannians the latter condition is equivalent to stable bi-rational equivalence. Indeed,
clearly, $X_N=G(Q,n)$ is stably bi-rationally equivalent to the flag variety $F(Q,n)$ of subspaces of dimensions from $0$ to $n$ (as the latter variety is a consecutive projective bundle over the former one).
And flag variety is rational as soon as it has a rational point (as a consecutive quadric fibration
$F(Q,n)\row F(Q,n-1)\row\ldots\row Q\row\op{Spec}(k)$). If $F(P,l)$ is the flag variety
corresponding to the motive $L$, then $F(Q,n)\times F(P,l)$ is stably bi-rationally equivalent to both $F(Q,n)$ and $F(P,l)$, since each
variety is rational over the generic point of the other (recall, that these varieties have rational points simultaneously). Thus, (1) implies (2). The opposite implication follows from
\cite[Theorem 4.17]{lens} taking into account \cite[Corollary 4.4]{lens}.

Finally, (2) $\Leftrightarrow$ (3) by \cite[Theorem 2.3.4]{IMQ} and above considerations, since for quadratic Grassmannians
the existence of a zero-cycle of degree 1 is equivalent to that of a rational point (by the Theorem of Springer).
\Qed
\end{proof}

Now we can describe the relations among $\ddet(Q)$ in $\picq/\ttt$. Any such relation can be reduced to the form
$\prod_i\ddet(P_i)=\prod_j\ddet(Q_j)$.

\begin{thm}
\label{relations}
Let $P_i$, $Q_j$ be some smooth projective quadrics.
The following conditions are equivalent:
\begin{itemize}
\item[$(1)$ ]
$\prod_i\ddet(P_i)=\prod_j\ddet(Q_j)\in\picq/\ttt$;
\item[$(2)$ ] $\oplus_iM(P_i)\stackrel{T}{\sim}\oplus_jM(Q_j)$.
\end{itemize}
\end{thm}

\begin{proof} $(2\Rightarrow 1)$
Let $\oplus_iM(P_i)=(\oplus Tates)\oplus\oplus_lN_l$
and $\oplus_jM(Q_j)=(\oplus Tates)\oplus\oplus_lM_l$
be the decompositions of both sides of (2) into Tates and anisotropic irreducibles.
Then (after reordering) we have isomorphisms $M_l\cong N_l(a_l)[2a_l]$, for some $a_l\in\zz$.
Then as $\frac{\Phi^E}{\Phi^k}$ ignores Tate-motives and Tate-shifts, we get:
\begin{equation*}
\begin{split}
&\frac{\Phi^E}{\Phi^k}\left(\prod_i\ddet(P_i)\right)=\frac{\Phi^E}{\Phi^k}(\ddet(\oplus_i M(P_i)))=
\frac{\Phi^E}{\Phi^k}(\ddet(\oplus_lN_l))=\\
&\frac{\Phi^E}{\Phi^k}(\ddet(\oplus_lM_l))=
\frac{\Phi^E}{\Phi^k}(\ddet(\oplus_j M(Q_j)))=\frac{\Phi^E}{\Phi^k}\left(\prod_j\ddet(Q_j)\right).
\end{split}
\end{equation*}
By Proposition \ref{inj-T}, $\prod_i\ddet(P_i)=\prod_j\ddet(Q_j)\in\picq/\ttt$.

$(1\Rightarrow 2)$ Let $\oplus_iM(P_i)=(\oplus Tates)\oplus\oplus_lN_l$
and $\oplus_jM(Q_j)=(\oplus Tates)\oplus\oplus_rM_r$ be the decomposition into a direct sum of Tates and anisotropic
irreducibles.
From (1) we know that
\begin{equation*}
\begin{split}
&\frac{\Phi^E}{\Phi^k}(\ddet(\oplus_lN_l))=\frac{\Phi^E}{\Phi^k}(\ddet(\oplus_i M(P_i)))=
\frac{\Phi^E}{\Phi^k}\left(\prod_i\ddet(P_i)\right)=\\
&\frac{\Phi^E}{\Phi^k}\left(\prod_j\ddet(Q_j)\right)=\frac{\Phi^E}{\Phi^k}(\ddet(\oplus_j M(Q_j)))=
\frac{\Phi^E}{\Phi^k}(\ddet(\oplus_rM_r)).
\end{split}
\end{equation*}
Let us cancel as many isomorphic (up to Tate-shift) direct summands from both sides as possible.
Suppose, the remaining relation is non-trivial.
So we can assume that, $N_l$ is not isomorphic to $M_r$ for any $l,r$, while we have:
$$
\frac{\Phi^E}{\Phi^k}(\ddet(\oplus_lN_l))=\frac{\Phi^E}{\Phi^k}(\ddet(\oplus_rM_r)).
$$
Each such irreducible is a direct summand in the (possibly, shifted) motive of some smooth projective anisotropic quadric.
To each such summand $N_l$ (respectively $M_r$), we can associate the field extension $F_l/k$ (respectively $E_r/k$) as follows.
Let $N_l$ be a direct summand of $M(R)$, s.t. the smallest Tate in $(N_l)_{\kbar}$ is $T(m)[2m]$. Then take $F_l=k(X_l)$ - the
function field of the Grassmannian $X_l=G(R,m)$ of $m$-dimensional projective subspaces on $R$, and similarly for $E_r=k(Y_r)$.

If any two of the above extensions are stably bi-rationally equivalent, then the respective indecomposable direct summands are
isomorphic (up to shift) by Proposition \ref{dir-sum-facts}.
Consider the directed graph whose vertices are isomorphism (up to shift)
classes of $\{N_l\}_l$, $\{M_r\}_r$ and where we have an arrow $N_l\row M_r$ (respectively, $N_l\row N_{l'}$, $M_r\row M_{r'}$)
iff there is a rational map $X_l\dashrightarrow Y_r$ ($X_l\dashrightarrow X_{l'}$, etc.). Since the existence of such a map is a
transitive property, our graph has no directed cycles (as we have chosen a single representative from each isomorphism class of
direct summands).
Hence, our graph has (at least one) final vertex. Suppose, it is $N_l$. Then, every indecomposable summand (from our list) which
is not isomorphic to $N_l$ will stay anisotropic over $F_l$.
Indeed, if $L$ is this other summand, and some Tate-motive would split from $L_{F_l}$ (by \cite[Theorem 4.19]{lens} we can always assume it
to be the "lower" one - see \cite[Appendix]{EC}), then the lowest Tate-motive (from which $L$ "starts")
will split there as well (since splitting from $L$ is equivalent to the splitting from the motive of the respective quadric $L$ is part of,
and for quadrics the splitting of the larger "lower" Tate-motive implies the splitting of the smaller one). But the lowest such Tate-motive
can't split from the motive of the mentioned quadric, since the respective Grassmannian has no rational point over $F_l$.
In particular, since no $M_r$'s were isomorphic to $N_l$, we get that all $M_r$'s stay anisotropic over $F_l$. Hence,
$\frac{\Phi^{F_l}}{\Phi^k}(\ddet(\oplus_rM_r))=T$. Similarly, $\frac{\Phi^{F_l}}{\Phi^k}(\ddet(N_{l'}))=T$, for all
$N_{l'}$ not isomorphic to
$N_l$.  At the same time, some (actually, exactly two) Tate-motives split from
$N_l$ over $F_l$, since the Tate-motive $T(m)[2m]$ splits from $M(R)|_{F_l}$.
The fact that the number of the "lower" Tate-summands split from $N_l|_{F_l}$ is equal to the number of
the "upper" Tate-summands split follows from \cite[Theorem 4.19]{lens} (while the fact that there is only one of each kind follows
from \cite[Theorem 4.17]{lens}, but we don't need this). Hence
$\frac{\Phi^{F_l}}{\Phi^k}(\ddet(N_l))=T(x)[y]\neq T$ (follows from the Definition \ref{def-PhiEk-detN} taking into account that no
"lower" Tate-motive of $N_l$ can be "above" an "upper" one (since it is so for the quadric)). Then
$\frac{\Phi^{F_l}}{\Phi^k}(\ddet(\oplus_nN_n))=T(d\cdot x)[d\cdot y]$, where $d$ is the number of indecomposables $N_{l'}$ isomorphic
(up to shift) to $N_l$. We obtain a contradiction: $T=T(d\cdot x)[d\cdot y]$. Thus, we can cancel all the terms, and so,
$\oplus_iM(P_i)\stackrel{T}{\sim}\oplus_jM(Q_j)$.
\Qed
\end{proof}

In particular, we can see that $\ddet(Q)\in\picq$ is a complete invariant of $M(Q)$.
This extends the Criterion of motivic equivalence of projective quadrics - \cite[Prop 5.1]{IMQ}, or \cite[Thm 4.18]{lens}
(see also \cite{Kar2}).

\begin{cor}
\label{mot-eq-proj}
Let $P$ and $Q$ be smooth projective quadrics. Then the following conditions are equivalent:
\begin{itemize}
\item[$(1)$ ] $M(P)\cong M(Q)$;
\item[$(2)$ ] $\ddet(P)=\ddet(Q)\in\picq$.
\end{itemize}
\end{cor}

\begin{proof}
$(1\Rightarrow 2)$ This is Proposition \ref{Phidet-MQ}(2).

$(2\Rightarrow 1)$  Since $\ddet(P_{\kbar})=\ddet(Q_{\kbar})$, we obtain that $\ddim(P)=\ddim(Q)$.
It follows from 2) and Theorem \ref{relations} that $M(P)\stackrel{T}{\sim} M(Q)$. Since $T$-equivalence preserves the
rank of anisotropic summand, and the total rank is the same (since dimensions are the same), we obtain that
the number of Tate-summands in $M(P_E)$ and $M(Q_E)$ is the same, hence, $i_W(Q_E)=i_W(P_E)$, for all $E/k$.
By the criterion of motivic equivalence, $M(P)\cong M(Q)$.
\Qed
\end{proof}

Let $\{X_l\}_l$ be the collection of all anisotropic quadratic Grassmannians for all quadrics from the set $\{P_i\}_i$
(considered with multiplicities), and
$\{Y_r\}_r$ be the collection of all anisotropic quadratic Grassmannians for all quadrics from the set $\{Q_j\}_j$.
Let $\{\hii_{X_l}\}_l$ be the collection of the motives of the respective \v{C}ech simplicial schemes $\check{C}ech(X_l)$
(again, with multiplicity),
and similarly for $\{\hii_{Y_r}\}_r$.
Repeating the arguments of the proof of \cite[Proposition 5.1]{IMQ} we obtain:

\begin{prop}
\label{hii-3}
The conditions $(1)$ and $(2)$ of Theorem \ref{relations} are equivalent to:
\begin{itemize}
\item[$(3)$ ] $\{\hii_{X_l}\}_l\cong\{\hii_{Y_r}\}_r$.
\end{itemize}
\end{prop}

\begin{proof}
Note, that by \cite[Theorem 3.1]{IMQ}, the anisotropic part of $\oplus_iM(P_i)$ is an extension of (shifted) $\hii$'s from
the union of the two copies of the set $\{\hii_{X_l}\}_l$ (corresponding to the "upper" and "lower" Tate-motives, respectively)
which I will denote "2"$\{\hii_{X_l}\}_l$, and similarly for the RHS.

Let $N_l$ be some indecomposable (anisotropic) irreducible summand of the $\oplus_{i}M(P_i)$, and $X_l$ be the respective
Grassmannian (corresponding to the
lowest Tate-motive in $N_l$). Then (3) implies that $\hii_{X_l}$ is isomorphic to some $\hii_{Y_r}$, and
it follows from \cite[Theorem 4.17]{lens} that $\oplus_{j}M(Q_j)$ contains an irreducible summand
isomorphic up to shift to $N_l$. By \cite[Theorem 3.1, Theorem 3.7]{IMQ}, $N_l$ is an extension of (shifted) motives of
\v{C}ech simplicial schemes
of Grassmannians and the respective $\hii$'s are determined uniquely by $N_l$ (since such a $\hii$ is trivial if and only if the respective
Tate-motive splits from $N_l$, and this information determines $\hii$ by \cite[Theorem 2.3.4]{IMQ}). Hence, we can identify the respective subsets in "2"$\{\hii_{X_l}\}_l$ and "2"$\{\hii_{Y_r}\}_r$.
Canceling $N_l$'s on both sides as well as the mentioned subsets, we reduce to smaller identical collections in "2"(3) and to shorter
sums of indecomposables in (2). Continuing this way, we cancel all the indecomposables in (2). This shows that (3) implies (2).
The converse is clear, since $N_l$ is an extension of $\hii$'s which are uniquely determined by $N_l$ as explained above. Thus, (2)
implies that "2"$\{\hii_{X_l}\}_l$ $\cong$ "2"$\{\hii_{Y_r}\}_r$.
\Qed
\end{proof}

Denote as $\cn(k)$ the set of isomorphism (up to shift) classes of indecomposable anisotropic direct summands in
the motives of smooth projective quadrics over $k$. Then Theorem \ref{relations} shows that we have an embedding of
$\picq/\ttt$ into a free abelian group generated by the set $\cn(k)$:
$$
\xymatrix{
\picq/\ttt\,\, \ar@{^{(}->}[r] & \,\oplus_{N\in\cn(k)}\zz\cdot\ddet(N)
}
$$
mapping $\ddet(Q)$ to the sum of determinants of anisotropic irreducible summands $M(Q)$ is a direct sum of (Tate-summands are ignored). Indeed,
since $\picq/\ttt$ is a quotient of a free abelian group generated by $\ddet(Q)$ where $Q$ runs over (isomorphism classes of)
smooth projective quadrics over $k$ modulo relations in (1) of Theorem \ref{relations}, this theorem
(together with Definition \ref{T-equiv} and the fact that $\ddet$ is stable under Tate-shift)
ensures that the map is well-defined and injective, since $\prod_i\ddet(P_i)/\prod_j\ddet(Q_j)=1$ in $\picq/\ttt$ if and only if
the respective element $\sum_{l}\ddet(N_l)-\sum_{r}\ddet(M_r)$ is zero in $\oplus_{N\in\cn(k)}\zz\cdot\ddet(N)$.
Note also that our map maps $\ddet(Q)$ the same way as the determinant of its anisotropic part. So, we get:

\begin{cor}
\label{descr-picq}
The group $\picq/\ttt$ is isomorphic to the image of the map:
$$
\ffi:\displaystyle\oplus_{Q\in\cq}\zz\cdot\ddet(Q)\lrow\oplus_{N\in\cn(k)}\zz\cdot\ddet(N),
$$
where $\cq$ is the set of isomorphism classes of smooth anisotropic projective quadrics over $k$.
In particular, $\picq$ is a free abelian group.
\end{cor}

It is an interesting question, if $\ffi$ is surjective, or not. It is related to an old and non-trivial
motivic question:

\begin{que} {\rm (\cite[Question 4.16]{lens})}
\label{old-question}
Let $Q$ be a smooth projective quadric and $N$ be an indecomposable direct summand in $M(Q)$.
Is it true that there exists a smooth projective quadric $P$ over $k$ with direct summand $M$ of $M(P)$,
such that $M$ is isomorphic to $N$ up to Tate-shift, and $M_{\kbar}$ contains $T$?
\end{que}

We get:

\begin{prop}
\label{conditional-sta}
Suppose, Question \ref{old-question} has a positive answer. Then
$$
\picq/\ttt\cong\oplus_{N\in\cn(k)}\zz\cdot\ddet(N).
$$
Alternatively, in $\picq/\ttt$ we can choose a basis $\{e^{q_i}\}_i$, where we take exactly one representative $q'_i$ in every
class of stable bi-rational equivalence of anisotropic quadratic forms over $k$.
\end{prop}

\begin{proof}
If the Question \ref{old-question} has a positive answer, then any $N\in\cn(k)$ will be represented by
a direct summand in $M(Q'_i)$, for some $i$, a direct summand starting from $T$ (over $\kbar$).
By \cite[Thm 4.13]{lens} (taking into account \cite[Cor 4.4]{lens} and Proposition \ref{dir-sum-facts})
the multiplicity (up to shift) of such a summand in
$M(Q'_i)$ is equal to the 1-st higher Witt index $i_1(q'_i)$ of our quadric - see \cite[Sect. 7]{lens}.
Since $Q_i$ is a subquadric of co-dimension 1 in $Q'_i$,
by \cite[Cor 4.9]{lens}, the multiplicity of $N$ in $M(Q'_i)$ will be exactly 1 more than that in $M(Q_i)$.
Indeed, recall that $size(N)=\op{max}(i|\Ch^i(N_{\kbar})\neq 0)-
\op{min}(i|\Ch^i(N_{\kbar})\neq 0)$ - \cite[Def. 4.6]{lens}. Then, either $i_1(q'_i)=1$, in which
case, $size(N)=\ddim(Q'_i)$ - \cite[Prop 4.5]{lens} and $M(Q_i)$ does not contain such summands at all, or $i_1(q'_i)>1$,
in which case, $N$ is also a direct summand of $M(Q_i)$ of multiplicity $i_1(q_i)=i_1(q'_i)-1$. Also, by the same results,
the remaining indecomposable direct summands of $M(Q_i)$ and $M(Q'_i)$ will be of strictly smaller
size. Indeed, the "outer shell" (all the Tate-motives split over the generic point) of $M(Q'_i)$ as well as that of
$M(Q_i)$ (if $i_1(q'_i)>1$) are completely covered by the copies of $N$, by \cite[Thm. 4.13]{lens}.
Hence, remaining direct summands start and end in "higher shells",
and so, have smaller $size$ - \cite[Cor. 4.14]{lens}.
Hence, $\ddet(Q'_i)=\ddet(N)^{i_1(q'_i)}\cdot \ddet(\text{of smaller motives})$, while
$\ddet(Q_i)=\ddet(N)^{i_1(q'_i)-1}\cdot \ddet(\text{of smaller motives})$ and so,
$e^{q_i}/\ddet(N)=\ddet(Q'_i)/(\ddet(Q_i)\cdot\ddet(N))$ is expressible in terms of $\ddet(M)$, for $size(M)<size(N)$.
By induction on $size(N)$, we get that the map $\ffi$ is surjective and the collection $\{e^{q_i}\}_i$ generates $\picq/\ttt$.
The fact that it is linearly independent follows from Theorem \ref{Dedekind}.
\Qed
\end{proof}

\begin{exa}
\label{real}
Let $k=\rr$. Then the only indecomposable anisotropic direct summands in the motives of real quadrics are Rost motives corresponding
to the pure symbols $\{-1\}^r,\,r\in\nn$, and so, it follows from Example \ref{Pfister-decomp} that the map $\ffi$ is surjective
and
\begin{equation*}
(\picq/\ttt)(\rr)=
\operatornamewithlimits{\oplus}_{r\in\nn}\zz\cdot e^{\lva -1\rva^{r}}.
\end{equation*}
\end{exa}

\section{An alternative to the method of Bachmann}
\label{alternative}

In this section I present an approach which can serve as an alternative to the one used by T.Bachmann.
Here I instead employ some ideas of \cite{IMQ}.
This permits to look at the same questions from a slightly different point of view. Still, there are certain
similarities in both approaches. At the same time, the new method is applicable not just to the category generated by the motives
of quadrics, but to arbitrary geometric motives and can be used for the study of the whole $\op{Pic}$ group of $\dmgmk$.
This approach in the end leads to the "motivic category of an extension" which has many remarkable properties and
permits to study Voevodsky motives "locally". It will be discussed in details in a separate paper.

Let $\dmkF{F}$ be Voevodsky triangulated category of motives with coefficients in an arbitrary commutative (unital) ring $F$.
Let $Q$ be smooth projective variety (not necessarily connected) over $k$. Let $\check{C}ech(Q)$ be the respective
\v{C}ech simplicial scheme, where $(\check{C}ech(Q))_n=Q^{\times n+1}$ with faces and degeneracy maps being partial diagonals
and partial projections. Denote its motive as $\hii_{Q}$.
We get the natural projection $\check{C}ech(Q)\row\op{Spec}(k)$, which gives a distinguished triangle $\Delta_Q$ in $\dmkF{F}$:
\begin{equation}
\label{dtQ}
\hii_{Q}\lrow T\lrow \whii_{Q}\lrow\hii_{Q}[1].
\end{equation}
Motives $\hii_{Q}$ and $\whii_{Q}$ are mutually orthogonal idempotents in $\dmkF{F}$:
$$
\hii_{Q}^{\otimes 2}\stackrel{\cong}{\row}\hii_{Q};\hspace{5mm}  \whii_{Q}^{\otimes 2}\stackrel{\cong}{\low}\whii_{Q};
\hspace{5mm}\hii_{Q}\otimes\whii_{Q}\cong 0.
$$
Denote as $\pi_Q:\dmkF{F}\row\dmkF{F}$ and $\widetilde{\pi}_Q:\dmkF{F}\row\dmkF{F}$ the projection functors given by $\otimes\hii_{Q}$
and $\otimes\whii_{Q}$, respectively. Then the image of $\pi_Q$ is the full localizing subcategory $\dmkQF{Q}{F}$ consisting of
objects which are stable under $\otimes\hii_{Q}$, while the image of $\widetilde{\pi}_Q$ is the full localizing subcategory
$\dmkQF{\widetilde{Q}}{F}$ consisting of objects which are killed by $\otimes\hii_{Q}$.
It follows from \cite[Thm 2.3.2]{IMQ} (which is, basically, \cite[Lem 4.9]{Vo-BKMK}) that
\begin{equation}
\label{polu-ort}
\Hom_{\dmkF{F}}(U,V)=0,\hspace{3mm}\text{for any}\,\,
U\in Ob(\dmkQF{Q}{F})\,\,\text{ and }\,\, V\in Ob(\dmkQF{\widetilde{Q}}{F}).
\end{equation}
At the same time, any object $W$ in $\dmkF{F}$
has a functorial decomposition:
$$
\pi_Q(W)\lrow W\lrow\widetilde{\pi}_{Q}(W)\lrow\pi_Q(W)[1].
$$
If $P$ and $Q$ are two smooth projective varieties over $k$, then there are natural identifications:
$$
\hii_{P\times Q}\stackrel{\cong}{\row}\hii_{P}\otimes\hii_{Q};\hspace{5mm}
\whii_{P\sqcup Q}\stackrel{\cong}{\low}\whii_{P}\otimes\whii_{Q}.
$$
The respective functors $\pi_Q,\widetilde{\pi}_Q,\pi_P,\widetilde{\pi}_P$ commute up to isomorphism, and
we have identifications of functors: $\pi_{Q\times P}\cong\pi_{Q}\circ\pi_{P}$ and
$\widetilde{\pi}_{Q\sqcup P}\cong\widetilde{\pi}_Q\otimes\widetilde{\pi}_P$.
It is also worth recalling (\cite[Thm 2.3.4]{IMQ}) that there is an identification
$\hii_{Q}\cong\hii_{P}$ if and only if $P$ has a zero-cycle of degree 1 over every generic point of $Q$, and vice-versa.

Let now ${\mathbf X}=\{X_i\}_{i\in I}$ be some finite collection of smooth projective varieties.
For any subset $J\subset I$ denote:
$$
\hii_{{\mathbf X}_J}:=\left(\otimes_{i\in J}\hii_{X_i}\right)\otimes\left(\otimes_{i\not\in J}\whii_{X_i}\right).
$$
Note, that these are still idempotents: $\hii_{{\mathbf X}_J}^{\otimes 2}\cong\hii_{{\mathbf X}_J}$.
Denote as $\pi_{{\mathbf X}_J}$ the respective projector. Its image is a full localizing subcategory $\dmkQF{{\mathbf X}_J}{F}$
of $\dmkF{F}$ made out of objects which are stable under $\otimes\hii_{X_i}$, for $i\in J$, and are killed by
$\otimes\hii_{X_i}$, for $i\not\in J$.

Tensoring the distinguished triangles $\Delta_{X_i}$ from (\ref{dtQ}) for every $X_i,\,i\in I$ we obtain what I will call a
{\it distinguished poly-triangle} $\Delta_{{\mathbf X}}$.
It presents the unit $T$ of the tensor structure of $\dmkF{F}$ as an extension of graded pieces $\hii_{{\mathbf X}_J}$, for all
$J\subset I$.
Note, that it follows from (\ref{polu-ort}) that, for $J_1,J_2\subset I$, we have:
\begin{equation}
\label{polu-poly-ort}
\Hom_{\dmkF{F}}(U,V)=0,\,\,\text{for}\,\,
U\in Ob(\dmkQF{{\mathbf X}_{J_1}}{F}),\,V\in Ob(\dmkQF{{\mathbf X}_{J_2}}{F}),\,\,\,\text{if}\,\,J_1\not\subset J_2.
\end{equation}
As any object $W$ of $\dmk$ is an extension of $\pi_{{\mathbf X}_J}(W)$, for $J\subset I$, we get:

\begin{prop}
\label{cons-piJ}
The functor:
$$
\times_{J\subset I}\pi_{{\mathbf X}_J}:\dmkF{F}\lrow\times_{J\subset I}\dmkQF{{\mathbf X}_{J}}{F}
$$
is conservative.
\end{prop}

Since $\pi_{{\mathbf X}_J}$ is symmetric monoidal, it preserves duals. Hence, using Proposition \ref{cons-piJ},
we obtain:

\begin{prop}
\label{dist-pic-XJ}
The functor:
$$
\times_{J\subset I}\pi_{{\mathbf X}_J}:\dmgmkF{F}\lrow\times_{J\subset I}\dmkQF{{\mathbf X}_{J}}{F}
$$
detects invertible objects.
\end{prop}

Moreover, we have:

\begin{prop}
\label{inj-pic-XJ}
The functor $\times_{J\subset I}\pi_{{\mathbf X}_J}$ is injective on the $\op{Pic}$.
\end{prop}

\begin{proof}

Denote as $T_P$, respectively $T_{P\times\wt{Q}}$, the unit objects of the category
$\dmkQF{P}{F}$, respectively, $\dmkQF{P\times\wt{Q}}{F}$.

\begin{lem}
\label{l1-inj}
Let
$P$ and $Q$ be smooth projective varieties, and $V\in\op{Pic}(\dmkQF{P}{F})$ be such an invertible object that
$V\otimes\hii_{Q}\cong T_{P\times Q}$ and $V\otimes\whii_{Q}\cong T_{P\times\widetilde{Q}}$.
Then $V\cong T_P$. Moreover, this identification $\otimes\hii_{Q}$ coincides with the original identification
of $V\otimes\hii_{Q}$ and $T_{P\times Q}$.
\end{lem}

\begin{proof}
Let $U\in\op{Pic}(\dmkQF{P}{F})$ be the inverse of $V$. Tensoring $V$ and $T_P$ with $\Delta_Q$ we get two exact triangles
in $\dmkQF{P}{F}$, where we can identify the $T_{P\times Q}$-terms:
\begin{equation}
\label{diag-TV}
\xymatrix{
T_P \ar[r] & T_{P\times\widetilde{Q}} \ar[r] & T_{P\times Q}[1] \ar[r] \ar@<1ex>[d]^{\ffi} & T_P[1] \\
V \ar[r] & T_{P\times\widetilde{Q}} \ar[r] & T_{P\times Q}[1] \ar[r] \ar@<1ex>[u]^{\psi} & V[1].
}
\end{equation}
Since $U=V^{-1}$ is invertible, we have:
$$
\Hom(T_{P\times\widetilde{Q}},V[1])=\Hom(U\otimes T_{P\times\widetilde{Q}},T_P[1])=\Hom(T_{P\times\widetilde{Q}},T_P[1]).
$$
By \cite[Thm 2.3.2]{IMQ} (or \cite[Lemma 4.9]{Vo-BKMK}), we can identify:
\begin{equation*}
\begin{split}
&\Hom(T_P,T_P[1])=\Hom(T_P,T[1])=\hm^{1,0}(\check{C}ech(P),F)=\op{H}^1_{Zar}(\check{C}ech(P),F),\\
&\Hom(T_{P\times Q},T_P[1])=\Hom(T_{P\times Q},T[1])=\hm^{1,0}(\check{C}ech(P\times Q),F)=\op{H}^1_{Zar}(\check{C}ech(P\times Q),F).
\end{split}
\end{equation*}
And since $\check{C}ech(P)$ and $\check{C}ech(P\times Q)$
in \'etale topology are both contractible to $\op{Spec}(k)$, by the Beilinson-Lichtenbaum Conjecture in weight zero (classical),
the groups $\hm^{1,0}(\check{C}ech(P),F)$ and $\hm^{1,0}(\check{C}ech(P\times Q),F)$ both embed into $\op{H}^{1}_{et}(\op{Spec}(k),F)$,
as the map from Zariski to \'etale cohomology is injective on the 1-st diagonal - \cite[Cor. 6.9]{VoMil}.
In particular, the natural map
$\Hom(T_P,T_P[1])\hookrightarrow\Hom(T_{P\times Q},T_P[1])$ is injective.
On the other hand, $\Hom(T_{P\times Q}[1],T_P[1])=\Hom(T_{P\times Q},T)=F$,
$\Hom(T_P[1],T_P[1])=\Hom(T_P,T)=F$ and the map $\Hom(T_P[1],T_P[1])\row\Hom(T_{P\times Q}[1],T_P[1])$ is an isomorphism.
Considering long exact sequences of $\Hom$'s from the triangle $T_{P\times Q}\row T_P\row T_{P\times\wt{Q}}\row T_{P\times Q}[1]$
to $T_P$ and $V$, we get that both groups: $\Hom(T_{P\times\widetilde{Q}},T_P[1])$ and $\Hom(T_{P\times\widetilde{Q}},V[1])$ are trivial.
Hence the above mutually inverse isomorphisms $\ffi,\psi$ can be extended to maps of exact triangles. In particular, we get maps
$\xymatrix{T_P \ar@<0.5ex>[r]^{f} & V \ar@<0.5ex>[l]^{g}}$. Moreover, since by (\ref{polu-ort}), there are no hom's from
$T_{P\times Q}[1]$ to $T_{P\times\widetilde{Q}}$, it follows that $(g\circ f)\otimes T_{P\times Q}=(\psi\circ\ffi)$.
Indeed, let $u:T_{P\times Q}[1]\row T_P[1]$, then $u\circ(\psi\circ\ffi)=(g\circ f)\circ u$. On the other hand,
$u\circ((g\circ f)\otimes T_{P\times Q})=(g\circ f)\circ u$, and so, the difference $(\psi\circ\ffi)-(g\circ f)\otimes T_{P\times Q}$
lifts to a map to $T_{P\times\widetilde{Q}}$ which must be zero.
Since
$\op{End}(T_{P})=F=\op{End}(T_{P\times Q})$ with the isomorphism (from left to right) given by $\otimes T_{P\times Q}$, and
$(\psi\circ\ffi)$ is invertible, we obtain that $(g\circ f)$ is invertible as well.
That means that $T_P$ is a direct summand in $V$. By \cite[Lem. 30]{Bac18}, this implies that $V\cong T_P$\footnote{I'm grateful to the
Referee for pointing out that connectivity of coefficients is not needed in this statement.}.
Alternatively, applying the same considerations to the bottom row of (\ref{diag-TV}) instead of the top one, one obtains that
$(f\circ g)\otimes T_{P\times Q}=(\ffi\circ\psi)$ and so, $(f\circ g)$ is invertible too, which gives the same isomorphism $V\cong T_P$.
\Qed
\end{proof}

\begin{lem}
\label{l2-inj}
If $P$ and $Q$ are smooth projective varieties, and $V\in\op{Pic}(\dmkF{F})$ be such that
$V\otimes\hii_{P}\cong T_P$, $V\otimes\hii_{Q}\cong T_Q$ and $V\otimes\hii_{P\times Q}\cong T_{P\times Q}$, and moreover,
these identifications commute with the maps $T_P\low T_{P\times Q}\row T_Q$. Then there is an identification
$V\otimes\hii_{P\sqcup Q}\cong T_{P\sqcup Q}$ and this commutes with the maps $T_P\row T_{P\sqcup Q}\low T_Q$.
\end{lem}

\begin{proof}
The natural projections provide a complex
$$
T_{P\times Q} \lrow T_P\oplus T_Q\lrow T_{P\sqcup Q}.
$$
I claim, that it extends to a distinguished triangle. Indeed, tenzoring it with $\hii_Q$ and $\whii_Q$, we get split
exact complexes $T_{P\times Q}\row T_{P\times Q}\oplus T_Q\row T_Q$ and $0\row T_{P\times\wt{Q}}\row T_{P\times\wt{Q}}$.
Since the unit $T$ is an extension of $\hii_Q$ and $\whii_Q$, we obtain the Mayer-Vietoris type distinguished triangle:
$$
T_{P\times Q} \lrow T_P\oplus T_Q\lrow T_{P\sqcup Q}\lrow T_{P\times Q}[1].
$$
Tensoring it with $V$ and identifying both in $T_{P\times Q}$ and $T_P\oplus T_Q$ terms
$$
\xymatrix{
T_{P\times Q} \ar[r] \ar[d]^{\cong}& T_P\oplus T_Q \ar[r] \ar[d]^{\cong} & T_{P\sqcup Q} \ar[r] & T_{P\times Q}[1] \ar[d]^{\cong}\\
T_{P\times Q} \ar[r] & T_P\oplus T_Q \ar[r] & V_{P\sqcup Q} \ar[r] & T_{P\times Q}[1]
}
$$
we extend it to an isomorphism of distinguished triangles. In particular, we get an identification
$T_{P\sqcup Q}\cong V_{P\sqcup Q}$ commuting with the needed maps.
\Qed
\end{proof}

Return to the proof of Proposition \ref{inj-pic-XJ}.
Let ${\mathbf X}^J=\prod_{i\in J}X_i$, and $T_{{\mathbf X}^J}=\hii_{{\mathbf X}^J}$.
Let us now prove by the decreasing induction on $|J|$ that $V\otimes T_{{\mathbf X}^J}\cong T_{{\mathbf X}^J}$ and these identifications
commute with the canonical maps $T_{{\mathbf X}^L}\row T_{{\mathbf X}^J}$, for $J\subset L$.

For $J=I$ we have $\hii_{{\mathbf X}^I}=\hii_{{\mathbf X}_I}$, and we get the identification from the conditions of
Proposition \ref{inj-pic-XJ}.

Suppose, we want to prove the inductive step for a given $J$. By the inductive assumption, we have (coherent) isomorphisms
for all $T_{{\mathbf X}^{J'}}$ with $J\subsetneqq J'$. It follows from the inductive application of Lemma \ref{l2-inj}
that we can add to this
coherent collection the isomorphism for $T_{{\mathbf X}^J\times(\sqcup_{i\not\in J}X_i)}$. Since
$T_{{\mathbf X}^J\times\widetilde{(\sqcup_{i\not\in J}X_i)}}=\hii_{{\mathbf X}_J}$, from the conditions of Proposition \ref{inj-pic-XJ}
and Lemma \ref{l1-inj}, we get that $T_{{\mathbf X}^{J}}$ can be added to our collection as well. Induction step is proven.
For $J=\emptyset$ we get an isomorphism $V\cong T$. Hence our map is injective on $\op{Pic}$.
\Qed
\end{proof}

To start with, let us reprove the result of T.Bachmann claiming that the reduced motives of affine quadrics are invertible.

\begin{prop} {\rm (Bachmann, \cite[Thm 33]{BQ})}
\label{eq-invert}
Let $A_q$ be an affine quadric $\{q=1\}$. Then the reduced motive $\widetilde{M}(A_q)$ is invertible in $\dmgmkD$.
\end{prop}

\begin{proof}
Let $\ddim(q)=n$ and $q'=\la 1\ra\perp -q$. Then $A_q=Q'\backslash Q$, and we have an exact triangle in $\dmkD$:
$$
M(Q')\row M(Q)(1)[2]\oplus T\row \widetilde{M}(A_q)[1]\row M(Q')[1].
$$
For a smooth projective quadric $R$ let us denote as $R^{i}$ the Grassmannian of the $i$-dimensional projective planes on $R$.
Since $Q$ is a co-dimension one subquadric of $Q'$, we have that, for any extension $E/k$, the following inequalities
on the Witt indices hold: $i_W(q|_E)\leq i_W(q'|_E)\leq i_W(q|_E)+1$. In particular, we have a chain of rational maps
between the respective Grassmannians:
\[\xymatrix@-0.1pc{Q'^{0} & Q^{0}\ar @{-->}[l] & Q'^{1}\ar @{-->}[l] & Q^{1}\ar @{-->}[l] & \ldots \ar @{-->}[l]   }\]
which induces
a chain of morphisms of Motives of their \v{C}ech simplicial schemes (cf. \cite[proof of Prop 2.3]{BV}):
$$
\hii_{Q'^{0}}\low\hii_{Q^{0}}\low\hii_{Q'^{1}}\low\hii_{Q^{1}}\low\hii_{Q'^{2}}\low\hii_{Q^{2}}\low\ldots.
$$
Let us rename it as:
$$
\hii_{X_1}\low\hii_{X_2}\low\hii_{X_3}\low\ldots\low\hii_{X_n}.
$$
Consider the collection ${\mathbf X}=\{X_i\}_{i=1}^n$. We can use the standard poly-binary approach as above, but since our
motives of \v{C}ech simplicial schemes are ordered, we can substitute it by an "ordered" version.
We have a Postnikov system in $\dmkD$:
$$
\xymatrix{
& \hii_{X_{0/1}} \ar[d]^{[1]} & \hii_{X_{1/2}} \ar[d]^{[1]} & & & \hii_{X_{n-1/n}} \ar[d]^{[1]} & \hii_{X_{n/n+1}} \ar[d]^{[1]}\\
T \ar[ru] & \hii_{X_1} \ar[l] \ar[ru] \ar@{}[lu]|-(0.2){\star}
& \hii_{X_2}  \ar[l]  \ar@{}[lu]|-(0.2){\star} & \ldots  & \hii_{X_{n-1}} \ar[ru] & \hii_{X_n} \ar[l] \ar[ru] \ar@{}[lu]|-(0.2){\star}
& 0 \ar[l]  \ar@{}[lu]|-(0.35){\star}
}
$$
where $\hii_{X_{i/i+1}}=\hii_{X_i}\otimes\whii_{X_{i+1}}$. Thus, the unit $T$ is an extension of idempotents $\hii_{X_{i/i+1}}$, $i=1,\ldots,n$. Note that $\hii_{X_i}\otimes\hii_{X_j}=\hii_{\op{max}(i,j)}$, while $\whii_{X_i}\otimes\whii_{X_j}=\whii_{\op{min}(i,j)}$,
and so $\hii_{X_{i/i+1}}\otimes\hii_{X_j}=\hii_{X_{i/i+1}}$, for $j\leq i$, and is zero otherwise.
Let $N'=n-1=\ddim(Q')$ and $N=n-2=\ddim(Q)$.
By \cite[Prop 3.6]{IMQ}, for a quadric $R$ of dimension $M$, we have an exact triangle in $\dmkD$:
\begin{equation*}
\oplus_{l=0}^{k}\hii_{R^{k/k+1}}(M-l)[2M-2l]\lrow M(R)\otimes\hii_{R^{k/k+1}}\lrow
\oplus_{l=0}^{k}\hii_{R^{k/k+1}}(l)[2l]\lrow\ldots,
\end{equation*}
where the maps are induced by the plane section cycles $T(M-l)[2M-2l]\row M(R)$ and the dual ones
$M(R)\row T(l)[2l]$.
(in loc. cit. it is stated for odd-dimensional quadrics only as motives were integral there, but the same arguments work for all
quadrics for motives with $\zz/2$-coefficients).
Considering $R=Q'$ and $R=Q$, and tensoring the above triangle with the appropriate $\hii_{X_{i/i+1}}$
(with $k=[i-1/2]$, resp. $k=[i-2/2]$),
we obtain exact triangles in $\dmkD$:
\begin{equation*}
\begin{split}
\oplus_{l=0}^{[i-1/2]}\hii_{X_{i/i+1}}(N'-l)[2N'-2l]\lrow M(Q')\otimes\hii_{X_{i/i+1}}\lrow
\oplus_{l=0}^{[i-1/2]}\hii_{X_{i/i+1}}(l)[2l]\lrow\ldots;\\
\oplus_{l=0}^{[i-2/2]}\hii_{X_{i/i+1}}(N-l)[2N-2l]\lrow M(Q)\otimes\hii_{X_{i/i+1}}\lrow
\oplus_{l=0}^{[i-2/2]}\hii_{X_{i/i+1}}(l)[2l]\lrow\ldots,
\end{split}
\end{equation*}
Thus $\widetilde{M}(A_q)\otimes\hii_{X_{i/i+1}}\cong\hii_{X_{i/i+1}}(i/2)[i-1]$, for $i$-even,
and $\cong\hii_{X_{i/i+1}}(N'-(i-1/2))[2N'-i+1]$, for $i$-odd.
In any case, $\widetilde{M}(A_q)\otimes\hii_{X_{i/i+1}}$ is invertible in $\dmkQD{X_{i/i+1}}$, for every $i$.
By the "ordered" version of Proposition \ref{dist-pic-XJ}, the functor
$$
\times(\otimes\hii_{X_{i/i+1}}):\dmgmkD\row\times_i\dmkQD{X_{i/i+1}}
$$
is conservative and detects invertible objects. Hence,
$\widetilde{M}(A_q)$ is invertible in $\dmkD$.
\Qed
\end{proof}

Denote as $Tate_{{\mathbf X}_{J}}=Tate(\dmkQF{{\mathbf X}_{J}}{F})$ the thick tensor triangulated subcategory of
$\dmkQF{{\mathbf X}_{J}}{F}$ generated by the Tate motives $T_{{\mathbf X}_{J}}(a)$.
Then $\op{Pic}(Tate_{{\mathbf X}_{J}})$ contains the subgroup $\ttt$ consisting of Tate-motives $T_{{\mathbf X}_{J}}(a)[b],\,a,b\in\zz$.

\begin{prop}
\label{z+z}
Suppose the category $\dmkQF{{\mathbf X}_{J}}{F}$ is non-zero. Then $\ttt\cong\zz\oplus\zz$.
\end{prop}

\begin{proof}
Suppose that $T_{{\mathbf X}_{J}}\cong T_{{\mathbf X}_{J}}(a)[b]$, for some $(a)[b]\neq (0)[0]$.

Our projector $\hii_{{\mathbf X}_{J}}$ has the form $\hii_P\otimes\whii_Q$ for some smooth projective $P$ and $Q$.
This projector will be non-zero if and only if $Q$ has no zero-cycle of degree 1 over the function field $E$ of
some connected component of $P$. Indeed, it follows from \cite[Thm 2.3.4, Thm 2.3.5]{IMQ} that if $Q$ has a zero-cycle of degree $1$
over each such function field, then $M(P)\otimes\whii_Q=0$. Since $\hii_P$ belongs to the localizing subcategory generated
by $M(P)$, $\hii_P\otimes\whii_Q$ is zero as well. For the converse, it is sufficient to restrict to $E$ and observe that
$\whii_Q|_E$ is non-zero by \cite[Thm 2.3.3]{IMQ}.
Hence, in our situation, $\hii_P\otimes\whii_Q|_E=\whii_{Q_E}$ is still non-zero.
Thus, by passing to $E$, we can assume that our projector is $\whii_Q$. As we have $(a)[b]$-periodicity, we obtain:
$$
0\neq \Hom(T_{\widetilde{Q}},T_{\widetilde{Q}})=\Hom(T_{\widetilde{Q}},T_{\widetilde{Q}}(a)[b])=
\Hom(T_{\widetilde{Q}},T_{\widetilde{Q}}(-a)[-b]).
$$
From the fact that $\Hom(T_Q,T_{\widetilde{Q}}(*)[*'])=0$, we have:
$\Hom(T_{\widetilde{Q}},T_{\widetilde{Q}}(c)[d])=\Hom(T,T_{\widetilde{Q}}(c)[d])$,
and we have an exact sequence
$\Hom(T,T_Q(c)[d+1])\low\Hom(T,T_{\widetilde{Q}}(c)[d])\low\Hom(T,T(c)[d])$.
The group $\Hom(T,T(c)[d])$ is zero, for $d>c$ and for $c<0$, while the group $\Hom(T,T_Q(c)[d+1])$ is zero, for $d\geq c$.
Thus, one of the groups: $\Hom(T_{\widetilde{Q}},T_{\widetilde{Q}}(a)[b])$, or $\Hom(T_{\widetilde{Q}},T_{\widetilde{Q}}(-a)[-b])$
will be zero - a contradiction. Hence, all the Tate-motives $T_{{\mathbf X}_{J}}(a)[b]$ are different.
\Qed
\end{proof}

For ${\mathbf X}=\{X_i\}_{i\in I}$, denote as $\un{2^I}$ the set of those $J\subset I$, for which the projector
$\hii_{{\mathbf X}_J}$ is non-zero.

\begin{cor}
\label{dist-TJ}
Let $A$ be an object of $\dmgmkF{F}$ and ${\mathbf X}=\{X_i\}_{i\in I}$ be some collection such that
$\pi_{{\mathbf X}_J}(A)\cong T_{{\mathbf X}_{J}}(a_J)[b_J]$, for every $J\subset I$. Then $A$ is invertible.
Moreover, two such objects are isomorphic if and only if the respective functions $(a_J)[b_J]:\un{2^I}\row(\zz)[\zz]$ are the same.
\end{cor}

\begin{proof}
It follows directly from Propositions \ref{dist-pic-XJ}, \ref{inj-pic-XJ} and \ref{z+z}.
\Qed
\end{proof}

Let now $A$ be some object of $\DQMgm$ (the category considered by T.Bachmann). Then $A$ is obtained from motives of finitely many
smooth projective quadrics $\{Q_l\}_{l\in L}$ using cone operations and tensor products (as well as direct summands and Tate-shifts).
Let $Q^r_l,\,r=0,\ldots,[\ddim(Q_l)/2]$ be the quadratic Grassmannians of $Q_l$. Consider the collection
${\mathbf X}(A)=\{X_i\}_{i\in I}$ of all these Grassmannians.

\begin{prop}
\label{decomp-tates}
Let $F=\zz/2$, and $A$ be an object of $\DQMgm$. Let ${\mathbf X}={\mathbf X}(A)$. Then
$\pi_{{\mathbf X}_J}(A)$ belongs to $Tate_{{\mathbf X}_{J}}$, for every $J\subset I$.
\end{prop}

\begin{proof}
From \cite[Prop 3.6]{IMQ} we have exact triangles in $\dmkD$:
\begin{equation*}
\oplus_{j=0}^{r}\hii_{Q_l^{r/r+1}}(N-j)[2N-2j]\lrow M(Q_l)\otimes\hii_{Q_l^{r/r+1}}\lrow
\oplus_{j=0}^{r}\hii_{Q_l^{r/r+1}}(j)[2j]\lrow\ldots,
\end{equation*}
where $\hii_{Q_l^{r/r+1}}=\hii_{Q_l^r}\otimes\whii_{Q_l^{r+1}}$ and $N=\ddim(Q_l)$.
Hence, $\pi_{Q_l^{r/r+1}}(M(Q_l))$ belongs to the Tate-motivic category.
And every projector $\pi_{{\mathbf X}_J}$ factors through some $\pi_{Q_l^{r/r+1}},\,r=0,\ldots,[\ddim(Q_l)/2]$ (for a fixed $l$).
Indeed, if $Q_l^t\in J$ and $Q_l^s\not\in J$, for some $t>s$, then the projector $\pi_{{\mathbf X}_J}$ contains the factor
$\hii_{Q_l^t}\otimes\whii_{Q_l^s}$, and so, is zero (as $\hii_{Q_l^t}\geq\hii_{Q_l^s}$ - see below).
Hence, for a non-zero projector (and fixed $l$), there exists $r$ such that
$Q_l^t\in J$, for $t\leq r$ and $Q_l^s\not\in J$, for $s>r$. Then $\pi_{{\mathbf X}_J}$ has the factor $\pi_{Q_l^{r/r+1}}$.
Thus, $\pi_{{\mathbf X}_J}(M(Q_l))$ belongs to $Tate_{{\mathbf X}_{J}}$, and so does $\pi_{{\mathbf X}_J}(A)$, as
$Tate_{{\mathbf X}_{J}}$ is closed under mentioned operations.
\Qed
\end{proof}

\begin{rem}
\label{z/2-coeff-expl}
The restriction on coefficients is caused by even-dimensional quadrics. For integral coefficients one has to use, in addition, Artin motives
corresponding to quadratic extensions. But if we restrict ourselves to odd-dimensional quadrics only, then the statement is true
with integral (and so, any) coefficients.
\end{rem}

For a collection ${\mathbf X}$, let the {\it reduced} collection $\un{{\mathbf X}}$ be the subset of those $X_i's$,
for which there are no zero-cycles of degree
$1$ over $k$. For a coefficient ring $F$ of prime characteristic, the collection is reduced exactly when $\emptyset\in\un{2^I}$.
If the collection ${\mathbf Y}$ contains ${\mathbf X}$, we call it a {\it refinement} of ${\mathbf X}$.

The following result shows that, at least, in working with $\picq$, both approaches are equivalent.
For $A\in\picq$, let ${\mathbf X}(A)={\mathbf X}$ be the collection, s.t. $\pi_{{\mathbf X}_J}(A)\in\ttt$, for all $J$.
We know from the proof of Proposition \ref{eq-invert} that such a collection exists.

\begin{prop}
\label{two-appr}
Let $A\in\picq$, and $\un{{\mathbf Y}}$ be any refinement of ${\mathbf X}(A)$ reduced. Then
$\Phi^k(A)=T(a)[b]\in\Kbt$ and $\pi_{\un{{\mathbf Y}}_{\emptyset}}(A)=T(a)[b]\in\dmkQF{\un{{\mathbf Y}}_{\emptyset}}{\zz/2}$,
for the same $(a)[b]$.
\end{prop}

\begin{proof}
It is sufficient to check it for $A=e^q$, where we get from Proposition \ref{phiE-on-ePQ} and the proof of Proposition \ref{eq-invert}
that in both cases
$$
(a)[b]=\sum_{l'=0}^{i_W(Q')-1}(m'-2l')[2m'-4l'+1]-\sum_{l=0}^{i_W(Q)-1}(m-2l)[2m-4l+1],
$$
where $q'=\la 1\ra\perp -q$, $m'=\ddim(Q')$ and $m=\ddim(Q)$.
\Qed
\end{proof}

Thus, to recover $\Phi^k(A)$, we don't need the whole collection ${\mathbf Y}$, but only $\pi_{\un{\mathbf Y}_{\emptyset}}$.
It is enough to take $Q=\sqcup_lQ_l$ the disjoint union of sufficiently many anisotropic quadrics, so that
$\whii_Q\otimes A\cong T_{\widetilde{Q}}(a)[b]$. Then $\Phi^k(A)\cong T(a)[b]$ as well.
Similarly, one can define $\Phi^E(A)$.

Let us introduce an ordering on \v{C}ech simplicial schemes:
$\hii_R\geq\hii_S$ iff the projection $\hii_R\row T$ factors through a map $\hii_R\row\hii_S$ in $\dmkD$.
This is equivalent to: the map $\hii_R\otimes\hii_S\row\hii_R$ is an isomorphism, or, which is the same,
$\whii_S\cong\whii_R\otimes\whii_S$. Since the automorphism group of $\whii_R$
is trivial, we can choose these identifications simultaneously for all inequalities in an associative way.
Then $\hii_R\cong\hii_S$ iff there are inequalities in both directions. Let $E$ be a finitely generated field extension, and $P$ be a
smooth projective variety with $k(P)=E$.
Let ${\mathbf{Q}}$ be the disjoint union of all connected varieties $Q$ with $\hii_Q\gneqq\hii_P$ (so, it is a smooth variety, but
with infinitely many components). Let $\hii_{{\mathbf{Q}}}$ be the motive of the respective \v{C}ech simplicial scheme, and
$\whii_{{\mathbf{Q}}}$ be the complementary projector.
Define the "motivic category of an extension":
\begin{equation}
\dmELD{E}{k}=\hii_P\otimes\whii_{{\mathbf{Q}}}\otimes\dmkD.
\end{equation}
Note, that the category $\dmELD{E}{k}$ is non-zero, and by the arguments from the proof of Proposition
\ref{z+z}, $\op{Pic}$ of it contains the subgroup of Tate-motives isomorphic to $\zz\oplus\zz$.
Let $\ffi_E:\dmkD\row\dmELD{E}{k}$ be the natural projection functor.
For $A\in\picq$, we know that $\ffi_E(A)\cong T_E(a)[b]$, for some (unique) $a,b\in\zz$. Then
$\Phi^E(A)\cong T(a)[b]$ as well. Indeed, we can reduce to the case of a trivial field extension using the natural functor
$$
\dmELD{E}{k}\row\dmELD{E}{E}.
$$
Now, all the calculations we did in Section \ref{structure-picq}, can be performed with the help of Corollary \ref{dist-TJ}
instead of \cite[Thm 31]{BQ}. Actually, the whole functor $\Phi^E$ of Bachmann can be alternatively introduced along these lines.
\\

\noindent
{\small {\bf Acknowledgements:} \ \
I'm grateful to Tom Bachmann for very useful discussions and to the Referee for very helpful suggestions which permitted to
improve the text.}

\end{document}